\documentclass[reqno]{amsart}
\usepackage{amsmath,amsthm,latexsym,amssymb,hyperref, amsfonts}
\usepackage{stmaryrd}
\usepackage{eucal}
\usepackage{mathrsfs}
\usepackage[all]{xy}
\usepackage{color}
\usepackage{tikz}
\usepackage{url}
\usepackage{fullpage}
\usepackage{tensor}
\usepackage{comment}
\usepackage{tikz-cd}
\usepackage{enumerate}
\usepackage{enumitem}
\usepackage{mathrsfs}
\usepackage{dutchcal}
\usetikzlibrary{matrix,arrows,decorations.pathmorphing}
\title[Coalgebras in the Dwyer-Kan Localization of a Model Category]{Coalgebras in the Dwyer-Kan Localization of a Model Category}

\author[Maximilien P\'eroux]{Maximilien P\'eroux}

\theoremstyle{definition}
\newtheorem{defi}{Definition}[section]

\newtheorem{rem}[defi]{Remark}

\numberwithin{equation}{section}

\long\def\emptytext#1{}

\theoremstyle{plain}
\newtheorem{thm}[defi]{Theorem}
\newtheorem{prop}[defi]{Proposition}

\newtheorem{cor}[defi]{Corollary}

\renewcommand{\o}{\otimes}
\newcommand{\bI}{\mathbb I}
\newcommand{\I}{\mathbb I}
\newcommand{\ai}{{\mathbb{A}_\infty}}
\newcommand{\ei}{{\mathbb{E}_\infty}}

\newcommand{\id}{\mathsf{id}}
\newcommand{\bS}{\mathbb{S}}
\newcommand{\bJ}{\mathbb{J}}

\newcommand{\Smash}{\wedge}
\newcommand{\sm}{\wedge}

\renewcommand{\r}{\rightarrow}
\newcommand{\si}{\Sigma^\infty}

\newcommand{\C}{\mathsf{C}}
\newcommand{\D}{\mathsf{D}}
\newcommand{\op}{^\mathsf{op}}
\newcommand{\M}{\mathsf{M}}

\newcommand{\N}{\mathscr{N}}

\newcommand{\Hom}{\mathsf{Hom}}

\newcommand{\ch}{\mathsf{Ch}}

\newcommand{\cofree}{\mathsf{T^\vee}}

\newcommand{\coalg}{\mathsf{CoAlg}}
\newcommand{\ccoalg}{\mathsf{CoCAlg}}
\newcommand{\comon}{\mathsf{CoAlg}}
\newcommand{\ccomon}{\mathsf{CoCAlg}}
\newcommand{\Mod}{\mathsf{Mod}}
\newcommand{\SpS}{\mathsf{Sp}^{{\Sigma}}}
\newcommand{\sset}{\mathsf{sSet}}
\newcommand{\Fin}{\mathsf{Fin}_*}
\newcommand{\W}{\mathsf{W}}

\newcommand{\Cinf}{\mathcal{C}}
\newcommand{\Oinf}{\mathcal{O}}
\newcommand{\coalginf}{\mathcal{CoAlg}}

\tikzset{
    labl/.style={anchor=south, rotate=-32, inner sep=.9mm}
}

\tikzset{
    labll/.style={anchor=south, rotate=32, inner sep=.9mm}
}

\begin{document}

\address{Department of Mathematics, University of Pennsylvania,
209 South 33rd Street,
Philadelphia, PA, 19104-6395, USA}
    \email{mperoux@sas.upenn.edu}

   

\begin{abstract}  
We show that weak monoidal Quillen equivalences induce equivalences of symmetric monoidal $\infty$-categories with respect to the Dwyer-Kan localization of the symmetric monoidal model categories. 
The result will induce a Dold-Kan correspondence of coalgebras in $\infty$-categories. Moreover it shows that Shipley's zig-zag of Quillen equivalences provides an explicit symmetric monoidal equivalence of $\infty$-categories for the stable Dold-Kan correspondence.
We study homotopy coherent coalgebras associated to a monoidal model category and we show examples when these coalgebras cannot be rigidified. That is, their $\infty$-categories are not equivalent to the Dwyer-Kan localizations of strict coalgebras in the usual monoidal model categories of spectra and of connective discrete $R$-modules.
\end{abstract}

\maketitle

\section{Introduction}

Let $\M$ be a model category and $\W$ its morphism class of weak equivalences. 
Recall that the homotopy category $\mathsf{Ho}(\M)$, associated to $\M$, is an ordinary category obtained by inverting all weak equivalences, and can also be denoted $\M[ \W^{-1}]$, see \cite[1.2.1, 1.2.10]{hovey}. However, the higher homotopy information is lost in $\mathsf{Ho}(\M)$. Dwyer and Kan, in \cite{dwyer-kan}, suggested instead a simplicial category $\mathsf{L}^H(\M, \W)$ sometimes called \emph{the hammock localization} of $\M$, that retains the higher information.
The idea has been translated into $\infty$-categories by Lurie in \cite[1.3.4.1, 1.3.4.15]{lurie1}. Following \cite{hin}, we shall prefer the term of\emph{ Dwyer-Kan localization} instead of \emph{underlying $\infty$-category} of a model category (see motivation by Remark \ref{rem: why call it DK loc?} below). If the model category is endowed with a symmetric monoidal structure compatible with the model structure, then the Dwyer-Kan localization is symmetric monoidal with respect to the derived tensor product.

The main result of this paper is Theorem \ref{thm: Weak monoidal eq imply strong in infinity} which shows that weak monoidal Quillen equivalences (as in \cite{monmodSS}) lift to equivalences of symmetric monoidal $\infty$-categories with respect to the Dwyer-Kan localizations. 

Any $\ai$-ring spectrum is homotopic to a strictly unital and associative ring spectrum, in some monoidal model category representing spectra, say symmetric spectra.
Similarly, any $\ei$-ring spectrum is homotopic to a strictly unital, associative and commutative ring spectrum. See \cite{SS}, \cite{MMSS}, \cite{MM} and \cite{EKMM}.
Associative and commutative algebras in the Dwyer-Kan localization of a symmetric monoidal model category $\M$ are precisely the $\ai$-algebras and $\ei$-algebras of $\M$, see  \cite[4.1.8.4, 4.5.4.7]{lurie1}.

In this paper, we answer the following question: can $\ai$-coalgebras and $\ei$-coalgebras  in spectra be homotopic to strictly counital, coassociative and cocommutative coalgebras over the sphere spectrum? In other words, can we \emph{rigidify} the comultiplication in spectra? 
We show in Corollary \ref{cor: non rigi  for coassociative} that $\ai$-coalgebras cannot be rigidified to strictly coassociative and counital coalgebras in spectra, following a previous result in \cite{perouxshipley}. 

A consequence of Theorem \ref{thm: Weak monoidal eq imply strong in infinity} is that it provides a Dold-Kan correspondence for $\ai$ and $\ei$-coalgebras in the Dwyer-Kan localization of  simplicial modules and non-negatively graded chain complexes. At the level of model categories, it was shown in \cite{sore3} that the Dold-Kan correspondence does not lift to a Quillen equivalence between simplicial coalgebras and chain coalgebras.
Therefore, we show in Corollary \ref{cor: failure of rigi} that $\ai$-coalgebras in connective modules over an Eilenberg-Mac Lane spectrum of a commutative ring do not correspond to their strict analogue in either simplicial modules or non-negatively graded chain complexes. Moreover, we show in Corollary \ref{cor: stable dold-kan eq} that the zig-zag of Quillen equivalences from \cite{hzalgshipley} provides Lurie's equivalence in the stable Dold-Kan correspondence {\cite[7.1.2.13]{lurie1}}.

In \cite{connectivecomod}, we are interested in comodules in the Dwyer-Kan localization of non-negatively graded chain complexes over a finite product of fields and show that homotopy coherent comodules can be rigidified to strict comodules.

The paper is constructed as follows. In Section \ref{section: dwyer kan} we recall the various definitions of Dwyer-Kan localizations in the literature in anticipation of our main result Theorem \ref{thm: Weak monoidal eq imply strong in infinity}.
In Section \ref{sec: rigidification setting}, we provide a comparison maps between $\ai$ and $\ei$ coalgebras and their strict analogues. We also provide a simple example where these maps are equivalences in Proposition \ref{prop: works for cartesian}.
In Section \ref{section: DK loc for coalgebras}, we apply Theorem \ref{thm: Weak monoidal eq imply strong in infinity} to the weak monoidal Quillen equivalences of \cite{monmodSS} and \cite{hzalgshipley} and show the failure of rigidification for connective discrete modules in Corollary \ref{cor: failure of rigi}. 
In Section \ref{sec: non rigi for spectra}, we provide a model structure for strict coalgebras in symmetric spectra following \cite{left2} and show in Theorem \ref{thm: perouxshipley result} that its Dwyer-Kan localization is not equivalent to the $\infty$-category of $\ai$-coalgebras in spectra.

\subsection*{Acknowledgment}The results here are part of my PhD thesis \cite{phd}, and as such, I would like to express my gratitude to my advisor Brooke Shipley for her help and guidance throughout the years. I would also like to thank Ben Antieau and Tasos Moulinos for helpful discussions in the early process of writing this paper.

\section{The Dwyer-Kan Localization of a Model Category}
\label{section: dwyer kan}

We show here our main result which is Theorem \ref{thm: Weak monoidal eq imply strong in infinity}. We recall the definition of Dwyer-Kan localization following \cite{hin}, \cite{tch} and \cite{lurie1}.

\subsection{The General Definition}\label{DKloc: general}
We begin with the general definition of the Dwyer-Kan localization of an $\infty$-category.

\begin{defi}[{\cite[1.3.4.1]{lurie1}}]\label{def: dk loc C[W-1] of any inf cat}
Let $\Cinf$ be an $\infty$-category and fix a collection  $\mathcal{W}\subseteq \mathsf{Hom}_{\mathsf{sSet}}(\Delta^1, \Cinf)$ of morphisms in $\Cinf$. The \emph{Dwyer-Kan localization of $\Cinf$ with respect to the collection $\mathcal{W}$} is an $\infty$-category, denoted $\Cinf [\mathcal{W}^{-1}]$, together with a functor $f:\Cinf\rightarrow \Cinf [\mathcal{W}^{-1}]$ that respects the following universal property. 
\begin{enumerate}
\item[(U)] For any other $\infty$-category $\mathcal{D}$, the functor $f$ induces an equivalence of $\infty$-categories:
\[
\begin{tikzcd}
\mathcal{Fun}( \Cinf [\mathcal{W}^{-1}] , \mathcal{D}) \ar{r}{\simeq} & \mathcal{Fun}^\mathcal{W}(\Cinf, \mathcal{D}),
\end{tikzcd}
\]
where $\mathcal{Fun}^\mathcal{W}(\Cinf, \mathcal{D})$ is the full subcategory of functors $\Cinf \rightarrow \mathcal{D}$ that sends morphisms in $\mathcal{W}$ to equivalences in $\mathcal{D}$.
\end{enumerate}
\end{defi}

The Dywer-Kan localization $\Cinf[\mathcal{W}^{-1}]$ always exists, for any choice of $\Cinf$ and $\mathcal{W}$, see \cite[1.3.4.2]{lurie1}, and is unique up to a contractible choice.
We shall be interested in the case when $\Cinf$ equals $\N(\M)$, the nerve of a model category $\M$. 

\begin{defi}[{\cite[1.3.4.15]{lurie1}}]
Let $\M$ be a model category and $\W$ its class of weak equivalences. 
We call $\N(\M)[\W^{-1}]$ the \emph{Dwyer-Kan localization of $\M$ with respect to $\W$} as in Definition \ref{def: dk loc C[W-1] of any inf cat}, where we abuse notation and let $\W$ denote the induced class of morphisms in the nerve $\N(\M)$.
\end{defi}

Notice that the homotopy category of $\N(\M)[{\W}^{-1}]$ is precisely the homotopy category $\mathsf{Ho}(\M)$ associated to $\M$ as in \cite[1.2.1, 1.2.10]{hovey}.

\begin{rem}\label{rem: why call it DK loc?}
We do not define the hammock localization $\mathsf{L}^H(\M, \W)$, but invite the interested reader to look for the explicit definition in \cite[2.1]{dwyer-kan}.
Since simplicial categories represent  $\infty$-categories,  a hammock localized simplicial category $\mathsf{L}^H(\M, \W)$ is a model for the Dwyer-Kan localization $\N(\M)[\W^{-1}]$. 
More precisely, by \cite[2.2.5.1]{htt}, there is a Quillen equivalence between the category of simplicial sets $\mathsf{sSet}$ endowed with the Joyal model structure and the category of simplicial categories $\mathsf{sCat}$ endowed with the Bergner model structure (see \cite{bergner}): 
\[
\begin{tikzcd}[column sep= huge]
\mathsf{sSet} \ar[shift left=2]{r}{\mathfrak{C}}[swap]{\perp} & \mathsf{sCat}.\ar[shift left=2]{l}{\mathfrak{N}}
\end{tikzcd}
\]
The functor $\mathfrak{N}:\mathsf{sCat} \rightarrow \mathsf{sSet}$ is the homotopy coherent nerve, or the simplicial nerve, as in \cite[1.1.5.5]{htt}. 
After a fibrant replacement, the functor $\mathfrak{N}$ sends  $\mathsf{L}^H(\M, \W)$ to the equivalence class of $\N(\M)[\W^{-1}]$,
as seen in \cite[1.3.1]{hin}.
\end{rem}

\begin{rem}\label{rem: DK loc on cof and fib are same}
As noted in \cite[1.3.4.16]{lurie1}, \cite[1.3.4]{hin}, and \cite[8.4]{dwyer-kan}, if the model category $\M$ admits \emph{functorial} fibrant and cofibrant replacement, in the sense of \cite[1.1.1. 1.1.3]{hovey}, then the following $\infty$-categories are equivalent:
\[
\N(\M_c)[\W^{-1}] \simeq \N(\M)[\W^{-1}] \simeq \N(\M_f)[\W^{-1}],
\]
where $\M_c\subseteq \M$ is the full subcategory of cofibrant objects, and $\M_f\subseteq \M$ is the full subcategory of fibrant objects.
\end{rem}

\subsection{Symmetric Monoidal Dwyer-Kan Localization}\label{DKloc: sym DK}

We now construct the symmetric mono\-idal structure on the Dwyer-Kan localization of a symmetric mono\-idal model category $\M$. This is a recollection of Appendix A in \cite{tch} and Section 4.1.7 on monoidal model categories in \cite{lurie1}.

\begin{defi}[{\cite[4.1.7.4]{lurie1}, \cite[A.4, A.5]{tch}}]\label{def: dk sym mon loc}
Let $\Cinf^\o$ be a symmetric mono\-idal $\infty$-category. Let $\mathcal{W}\subseteq \mathsf{Hom}_{\mathsf{sSet}}(\Delta^1, \Cinf)$ be a class of edges in $\Cinf$ that is stable under homotopy, composition and contains all equivalences. 
Suppose further that $\o: \Cinf \times \Cinf \rightarrow \Cinf$ preserves the class $\mathcal{W}$ separately in each variable. 
The \emph{symmetric monoidal Dywer-Kan localization of $\Cinf^\o$ with respect to $\mathcal{W}$} is a symmetric monoidal $\infty$-category, denoted $\Cinf[\mathcal{W}^{-1}]^\o$, together with a symmetric monoidal functor $i:\Cinf^\o\rightarrow\Cinf[\mathcal{W}^{-1}]^\o$ which is characterized by the following universal property.
\begin{enumerate}
\item[(U)]For any other symmetric monoidal $\infty$-category $\mathcal{D}^\o$, the functor $i$ induces an equivalence of $\infty$-categories: 
\[
\mathcal{Fun}_\o(\Cinf[\mathcal{W}^{-1}]^\o, \mathcal{D}^\o) \stackrel{\simeq}\longrightarrow \mathcal{Fun}_\o^\mathcal{W}(\Cinf^\o, \mathcal{D}^\o),
\]
where $\mathcal{Fun}_\o^\mathcal{W}(\Cinf^\o, \mathcal{D}^\o)$ is the full subcategory of symmetric monoidal functors $\Cinf^\o\rightarrow \mathcal{D}^\o$ that sends $\mathcal{W}$ to equivalences.
\end{enumerate}
\end{defi}

As noticed in \cite[A.5]{tch}, the underlying $\infty$-category of the symmetric monoidal category $\Cinf[\mathcal{W}^{-1}]^\o$ is precisely the Dwyer-Kan localization of $\Cinf$ with respect to $\mathcal{W}$ in the sense of Definition \ref{def: dk loc C[W-1] of any inf cat}, i.e.:
\[
\Big(\Cinf[\mathcal{W}^{-1}]^\o\Big)_{\langle 1 \rangle} \simeq \Cinf [\mathcal{W}^{-1}].
\]

\begin{rem}\label{rem: W edges are over id}
Let $\Cinf^\o$ and $\mathcal{W}$ be as in Definition \ref{def: dk sym mon loc}. Given the symmetric monoidal structure $\Cinf^\o\rightarrow \N(\Fin)$, products of $n$ edges in $\mathcal{W}$ in $\Cinf$ correspond precisely, under the equivalence:
\[
\Cinf^{\times n} \simeq \Cinf^\o_{\langle n \rangle},
\]
to morphisms lying over $\id_{\langle n \rangle}$ in $\N(\Fin)$. This defines a class of edges $\mathcal{W}^\o\subseteq \mathsf{Hom}_{\mathsf{sSet}}(\Delta^1, \Cinf^\o)$. Then the Dwyer-Kan localization of $\Cinf^\o$ with respect to $\mathcal{W}^\o$, in the sense of Definition \ref{def: dk loc C[W-1] of any inf cat}, denoted $\Cinf^\o\left[{\left(\mathcal{W}^\o\right)}^{-1}\right]$, is equivalent to $\Cinf[\mathcal{W}^{-1}]^\o$ defined above.
\end{rem}

We would like to study the case where the underlying $\infty$-category of $\Cinf^\o$ is the Dwyer-Kan localization $\N(\M)[\W^{-1}]$ of a model category $\M$.
We first recall the induced symmetric monoidal structure on the nerve of a symmetric monoidal category.

\begin{defi}[{\cite[2.0.0.1]{lurie1}}]\label{def: C otimes of a sym mon cat from lurie}
Let $(\C, \o, \I)$ be a symmetric monoidal category. Define a new category $\C^\o$, called \emph{the operator category of $\C$}, as follows. 
\begin{itemize}
\item Objects are sequences $(C_1, \ldots, C_n)$ where each $C_i$ is an object in $\C$, for all $1\leq i \leq n$, for some $n \geq 1$. We allow the case $n=0$ and thus the empty set $\emptyset$  as a sequence.

\item A morphism $(C_1, \ldots, C_n) \rightarrow (C'_1, \ldots, C'_m)$ in $\C^\o$ is a pair $(\alpha, \{ f_j \})$, where $\alpha$ is a map of finite sets $\alpha: \langle n \rangle \rightarrow \langle m \rangle $ and $\{f_j \}$ is a collection of $m$-morphisms in $\C$:
\[
f_j:\bigotimes_{i\in\alpha^{-1}(j)} C_i \longrightarrow C'_j,
\]
for all $1 \leq j \leq m$. If $\alpha^{-1}(j)=\emptyset$, then $f_j$ is a morphism $\I\rightarrow C_j'$.
\item The composition of morphisms in $\C^\o$ is defined using the compositions in $\Fin$ and $\C$ together with the associativity of the symmetric monoidal structure of $\C$.
\item The identity morphism on an object $(C_1, \ldots, C_n)$ is given by the identities in $\Fin$ and $\C$: $(\id_{\langle n \rangle}, \{\id_{C_j}\})$.
\end{itemize}
\end{defi}

We obtain a functor:
\[
\C^\o \longrightarrow \Fin,
\]
that sends $(C_1, \ldots C_n)$ to $\langle n \rangle$. The induced functor $\N(\C^\o) \rightarrow \N(\Fin)$ in $\infty$-categories is coCartesian and defines a symmetric monoidal structure.

\begin{prop}[{\cite[2.1.2.21]{lurie1}}]\label{prop: N(C) is a sym mon cat}
Given $(\C, \o, \I)$ a symmetric monoidal category, let $\C^\o$ be the operator category of $\C$. Then the nerve $\N(\C^\o)$ is a symmetric monoidal $\infty$-category whose underlying $\infty$-category is $\N(\C)$.
\end{prop}

\begin{rem}\label{rem: explicit coCart lift for C^o}
Let us provide the coCartesian lifts for $\N(\C^\o)\rightarrow \N(\Fin)$. Given $(C_1, \ldots, C_n)$ in $\C^\o$, and $\alpha: \langle n \rangle  \rightarrow \langle m \rangle$ a map in $\Fin$, the associated coCartesian lift is induced by defining $C'_j$ as follows:
\[
C'_j:= \bigotimes_{i\in \alpha^{-1}(j)} C_i,
\]
for each $1\leq j \leq m$. Define $C_j'=\I$ whenever $\alpha^{-1}(j)=\emptyset$. This determines a morphism $(C_1, \ldots, C_n) \rightarrow (C'_1, \ldots, C'_m)$ in the operator category $\C^\o$ as desired.
\end{rem}

If the symmetric monoidal category $(\C, \o, \I)$ happens to  be endowed with a model structure, the bifunctor $\o: \C\times \C \rightarrow \C$ need not preserve weak equivalences in either variable. We need to restrict to \emph{(symmetric) monoidal model categories} as in \cite[4.2.6]{hovey}.
In any  monoidal model category $(\M, \o, \I)$, the tensor $\o:\M\times \M \rightarrow \M$ preserves weak equivalences in each variable, if we restrict to cofibrant objects $\M_c\subseteq \M$. Moreover. the tensor product of cofibrant objects is again cofibrant. In model categories, this allows us to define a \emph{derived tensor product} for the homotopy category $\mathsf{Ho}(\M)=\M[\W^{-1}]$, see \cite[4.3.2]{hovey}. 
In higher category, the transition between the tensor product and the derived tensor product is exactly through the Dwyer-Kan localization of a symmetric monoidal $\infty$-category as in Definition \ref{def: dk sym mon loc}. 
If we suppose in addition that $\I$ is cofibrant, then, as in Definition \ref{def: C otimes of a sym mon cat from lurie}, we can define $\M_c^\o\subseteq \M^\o$ from the full subcategory of cofibrant objects $\M_c\subseteq \M$, since $(\M_c, \o, \I)$ is symmetric monoidal.

\begin{prop}[{\cite[4.1.7.6]{lurie1}, \cite[A.7]{tch}}]\label{prop: compatibility of DK in sym mon}
Let $(\M, \o, \I)$ be a symmetric monoidal model category. Suppose that $\I$ is cofibrant. 
Then the Dwyer-Kan localization $\N(\M_c)[{\W}^{-1}]$ of $\M$ can be given the structure of symmetric monoidal $\infty$-category via the symmetric monoidal Dwyer-Kan localization of $\N(\M_c^\o)$ in the sense of \textup{Definition \ref{def: dk sym mon loc}},
\[
\begin{tikzcd}
\N(\M_c^\o) \ar{r} & \N(\M_c)[{\W}^{-1}]^\o,
\end{tikzcd}
\] 
where ${\W}$ is the class of weak equivalences restricted to cofibrant objects in $\M$.
\end{prop}

\begin{rem}\label{rem:Dk SYM loc on cof or not is the same }
We warn the reader of the following technicality.
The inclusion of cofibrant objects $\M_c\subseteq \M$ induces a \emph{lax} symmetric monoidal functor $\N(\M_c^\o)\rightarrow\N(\M^\o)$.
From Remark \ref{rem: DK loc on cof and fib are same}, Proposition \ref{prop: compatibility of DK in sym mon} implies that we can also construct a symmetric monoidal $\infty$-category  $\N(\M)[\W^{-1}]^\o$ whose fiber over $\langle 1 \rangle$ is precisely $\N(\M)[\mathsf{W}^{-1}]$.
However, cofibrant replacement induces a functor $\N(\M^\o)\rightarrow \N(\M)[{\W}^{-1}]^\o$ that is only lax symmetric monoidal and does not share the same properties of universality as in Definition \ref{def: dk sym mon loc}. We invite the interested reader to look at \cite[A.7]{tch} for more details.
\end{rem}

\subsection{Weak Monoidal Quillen Equivalence}\label{DKloc: QE}
Given model categories $\C$ and $\D$, denote $\W_\C$ and $\W_\D$ their respective class of weak equivalences. Suppose we are given a Quillen adjunction:
 \[\begin{tikzcd}
L:\C \ar[shift left=2]{r} & \ar[shift left=2]{l}[swap]{\perp} \D:R.
\end{tikzcd}\] 
Then as the left adjoint functor $L$ preserves weak equivalences between cofibrant objects and the right adjoint functor $R$ preserves weak equivalences between fibrant objects, we get, by \cite[1.5.1]{hin}, a pair of adjoint functors in $\infty$-categories between the Dwyer-Kan localizations of $\C$ and $\D$:
\[
\begin{tikzcd}
\mathbb{L}:\N(\C)[\W^{-1}_\C] \ar[shift left=2]{r} & \ar[shift left=2]{l}[swap]{\perp} \N(\D)[\W^{-1}_\D]:\mathbb{R},
\end{tikzcd}
\] 
where $\mathbb{L}$ and $\mathbb{R}$ represent the derived functors of $L$ and $R$.
If $\C$ and $\D$ are symmetric monoidal model categories, we investigate when the derived functors are symmetric monoidal functors of $\infty$-categories.

\begin{defi}[{\cite[3.6]{monmodSS}}]\label{def: weak mon qui pair}
Let $(\C, \o, \bI)$ and $(\D, \sm, \bJ)$ be symmetric monoidal model categories. 
A \emph{weak monoidal Quillen pair} consists of a Quillen adjunction:
\[\begin{tikzcd}
L:(\C, \o, \bI) \ar[shift left=2]{r} & \ar[shift left=2]{l}[swap]{\perp} (\D, \sm, \bJ):R,
\end{tikzcd}
\]
where $L$ is lax comonoidal such that the following two conditions hold.
\begin{enumerate}[label=(\roman*)]
\item\label{enum: weak monoidal} For all cofibrant objects $X$ and $Y$ in $\C$, the comonoidal map:
\[
\begin{tikzcd}
L(X\o Y) \ar{r} & L(X)\Smash L(Y),
\end{tikzcd}
\]
is a weak equivalence in $\D$.
\item\label{enum; weak monoidal 2} For some (hence any) cofibrant replacement $\lambda:c\bI \stackrel{\sim}\longrightarrow \bI$ in $\C$, the composite map:
\[
\begin{tikzcd}
L(c\mathbb{I}) \ar{r}{L(\lambda)} & L(\I) \ar{r} & \bJ,
\end{tikzcd}
\]
is a weak equivalence in $\D$, where the unlabeled map is the natural comono\-idal structure of $L$.
\end{enumerate}
A weak monoidal Quillen pair is a \emph{weak monoidal Quillen equivalence} if the underlying Quillen pair is a Quillen equivalence.
\end{defi}

\begin{thm}\label{thm: Weak monoidal eq imply strong in infinity}
Let $(\C, \o, \bI)$ and $(\D, \sm, \bJ)$ be symmetric monoidal model categories with cofibrant units. 
Let $\W_\C$ and $\W_\D$ be the classes of weak equivalence in $\C$ and $\D$ respectively.
Let: \[\begin{tikzcd}
L:(\C, \o, \bI) \ar[shift left=2]{r} & \ar[shift left=2]{l}[swap]{\perp} (\D, \sm, \bJ):R,
\end{tikzcd}\] be a weak monoidal Quillen pair.
Then the derived functor of  $L:\C \r \D$ induces a symmetric monoidal functor between the Dwyer-Kan localizations: 
\[
\begin{tikzcd}
\mathbb{L}:\N (\C_c) \left[\W_\C^{-1}\right] \ar{r} & \N(\D_c) \left[\W_\D^{-1}\right],
\end{tikzcd}
\]
where $\C_c\subseteq \C$ and $\D_c\subseteq \D$ are the full subcategories of cofibrant objects. It is the left adjoint of the derived functor of $R:\D\rightarrow \C$.
If $L$ and  $R$ form a weak monoidal Quillen equivalence, then $\mathbb{L}$ is a symmetric monoidal equivalence of $\infty$-categories.
\end{thm}

\begin{proof}
Let $\C_c^\o$ and $\D_c^\o$ be the operator categories (Definition \ref{def: C otimes of a sym mon cat from lurie}). 
Denote the symmetric monoidal Dwyer-Kan localizations (Definition \ref{def: dk sym mon loc}) by:
\[
\begin{tikzcd}
i_\C:\N(\C_c^\o) \ar{r} & \N(\C_c)[{\W}_\C^{-1}]^\o, & i_\D:\N(\D_c^\o) \ar{r} & \N(\D_c)[{\W}_\D^{-1}]^\o,
\end{tikzcd}
\]
and denote their coCartesian fibrations by:
\[
\begin{tikzcd}
p: \N(\C_c)[{\W}_\C^{-1}]^\o \ar{r} & \N(\Fin), & q:  \N(\D_c)[{\W}_\D^{-1}]^\o \ar{r} & \N(\Fin).
\end{tikzcd}
\]
The functor $L:\C\rightarrow \D$, as a left Quillen functor, defines $\N(\C_c)\rightarrow \N(\D_c)$, and hence a functor $L^\o: \N(\C_c^\o) \rightarrow \N(\D_c^\o)$ that is compatible with the coCartesian structures:
\[
\begin{tikzcd}
 \N(\C_c^\o) \ar{r}{L^\o} \ar{dr} & \N(\D_c^\o)\ar{d} \ar{r}{i_\D}&  \N(\D_c)[{\W}_\D^{-1}]^\o\ar{dl}{q}\\
 & \N(\Fin).& 
\end{tikzcd}
\]
We show that the composite:
\[
\begin{tikzcd}
\N(\C_c^\o) \ar{r}{L^\o} & \N(\D_c^\o) \ar{r}{i_\D} & \N(\D_c)[{\W}_\D^{-1}]^\o
\end{tikzcd}
\]
is a symmetric monoidal functor that sends ${\W}_\C$ to equivalences, i.e., belongs to the $\infty$-category: \[\mathcal{Fun}^{{\W}_\C}_\o(\N(\C_c^\o),  \N(\D_c)[{\W}_\D^{-1}]^\o),\] as in Definition \ref{def: dk sym mon loc}. The latter is clear as $L$ is a left Quillen functor.
We are left to show that the composite sends $p$-coCartesian lifts to $q$-coCartesian lifts.
Let $(C_1, \ldots, C_n)$ be an object of $\C_c^\o$, and let $\alpha: \langle n \rangle \rightarrow \langle m \rangle$ be a morphism of finite sets. 
A $p$-lift of $(\alpha, (C_1, \ldots, C_n))$ is given in Remark \ref{rem: explicit coCart lift for C^o} by a certain sequence $(C'_1, \ldots, C'_m)$ in $\C_c^\o$, i.e., the induced map $(C_1, \ldots, C_n)\rightarrow (C'_1, \ldots, C'_m)$ is sent to $\alpha$ via the coCartesian functor $p$.
Since $L$ is weak monoidal functor, from \ref{enum: weak monoidal} of Definition \ref{def: weak mon qui pair}, we get that:
\[
\begin{tikzcd}
\displaystyle\bigwedge_{i\in \alpha^{-1}(j)} L(C_i) & \displaystyle\ar{l}[swap]{\sim} L \left( \bigotimes_{i\in \alpha^{-1}(j)} C_i \right)= L(C'_j),
\end{tikzcd}
\]
is a weak equivalence in $\D$, for all $1\leq j \leq m$. 
In the case $\alpha^{-1}(j)=\emptyset$, we apply \ref{enum; weak monoidal 2} of Definition \ref{def: weak mon qui pair} to obtain a weak equivalence:
\[
\begin{tikzcd}
\mathbb{J} & L(\bI)=L(C_j). \ar{l}[swap]{\sim} 
\end{tikzcd}
\]
Applying the localization $i_D$ and Remark \ref{rem: W edges are over id}, we get that $(L(C'_1), \ldots, L(C'_m))$ defines the desired $q$-coCartesian lift.

By the universal property (U) of the symmetric monoidal Dwyer-Kan localization in Definition \ref{def: dk sym mon loc}, the composite functor $i_\D\circ L^\o$ represents a symmetric monoidal $\infty$-functor:
\[
\begin{tikzcd}
\mathbb{L}^\o:\N(\C_c)[{\W}_\C^{-1}]^\o\ar{r} & \N(\D_c)[{\W}_\D^{-1}]^\o.
\end{tikzcd}
\]
Fiberwise over $\N(\Fin)$, the functor $\mathbb{L}^\o$ is precisely the product of the derived left adjoint functor $\mathbb{L}:\N(\C_c)[\W_\C^{-1}] \rightarrow \N(\D_c)[\W_\D^{-1}]$.
In particular, if $L$ is a Quillen equivalence, then $\mathbb{L}$ is an equivalence of $\infty$-categories, and hence $\mathbb{L}^\o$ is an equivalence of symmetric monoidal $\infty$-categories.
\end{proof}

\begin{rem}\label{rem: weak monoidal eq imply eq of coalgebras}
In \cite[3.12]{monmodSS}, Schwede and Shipley show that given a weak monoidal Quillen equivalence: \[\begin{tikzcd}
L:(\C, \o, \bI) \ar[shift left=2]{r} & \ar[shift left=2]{l}[swap]{\perp} (\D, \sm, \bJ):R,
\end{tikzcd}\] with cofibrant units, then the right adjoint $R$ induces Quillen equivalences between the category of monoids $\mathsf{Mon}(\D)$ and $\mathsf{Mon}(\C)$.
Our Theorem \ref{thm: Weak monoidal eq imply strong in infinity} strengthen the results when we work with $\infty$-categories. 
In particular, given any $\infty$-operad $\Oinf^\o$,
we get an equivalence of $\infty$-categories:
\[\mathcal{Alg}_\Oinf \left( \N(\C_c)[\W_\C^{-1}] \right)\simeq \mathcal{Alg}_\Oinf \left( \N(\D_c)[\W_\D^{-1}] \right),\]
which has been challenging to prove in the case of $\Oinf=\mathbb{E}_\infty$ in the past, see for instance \cite{richtershipley} and \cite[1.3, 1.4]{mandell}.
We also obtain an equivalence on the $\infty$-categories of coalgebras (see definition in \cite[2.1]{coalgenr}):
\[
\mathcal{CoAlg}_\Oinf \left( \N(\C_c)[\W_\C^{-1}] \right)\simeq \mathcal{CoAlg}_\Oinf \left( \N(\D_c)[\W_\D^{-1}] \right).
\]
Such a result on coalgebras has been shown to be untrue in some model categories, see for instance \cite{sore3} and Section \ref{subsection: rigi failure in connective} below.
\end{rem}

\section{The Rigidification Problem}\label{sec: rigidification setting}

In this section, we want to compare homotopy coherent coalgebras  with their strict analogues. On one hand, given a nice enough symmetric monoidal model category $\M$, we can obtain its Dwyer-Kan localization which is a symmetric monoidal $\infty$-category. We can then define the $\infty$-category of $\ai$ or $\ei$-coalgebras as in \cite[2.1]{coalgenr}. On the other hand, we can consider coalgebras (also called comonoids) in the monoidal category $\M$ in the classical sense and then take their Dwyer-Kan localization as in Definition \ref{def: dk loc C[W-1] of any inf cat}. We are interested to know if the $\infty$-categories are equivalent. 

There are classical rigidification results that compare $\mathbb{A}_\infty$-algebras with their strict associative analogue, see \cite[4.1.8.4]{lurie1}. There is also a comparison between the $\mathbb{E}_\infty$-case with the commutative case in \cite[4.5.4.7]{lurie1}. However, there is no reason to expect that these results dualize in general. In particular, if $\mathbb{A}_\infty$-algebras can be rigidified to strict associative algebras in a model category $\M$, there is no reason to expect that $\mathbb{A}_\infty$-coalgebras correspond to strict coassociative coalgebras in $\M$ as we shall see in Corollary \ref{cor: failure of rigi} and Theorem \ref{thm: perouxshipley result} below.

\subsection{Rigidification Properties}
Let $\C$ be a (symmetric) monoidal category and denote $\C^\o$ its operator category, as in Definition \ref{def: C otimes of a sym mon cat from lurie}. 
Let $p:\C^\o\rightarrow \Delta\op$ be its associated Grothendieck opfibration (see \cite[4.5]{groth}) that determines the monoidal structure of $\C$, and induces the coCartesian fibration $\N(\C^\o)\rightarrow \N(\Delta\op)$. 
There is a correspondence between monoids in $\C$ and sections of $p$ that send convex morphisms to $p$-coCartesian arrows (see \cite[4.21]{groth}).
In particular, we obtain the following equivalence of $\infty$-categories:
\[
\begin{tikzcd}
\N(\mathsf{Mon}(\C)) \ar{r} & \mathcal{Alg}_\ai(\N(\C)).
\end{tikzcd}
\]
By using opposite categories, we obtain therefore an equivalence:
\[
\begin{tikzcd}
\N(\comon(\C)) \ar{r} & \coalginf_\ai(\N(\C)).
\end{tikzcd}
\]
Let $\M$ be a symmetric monoidal model category with cofibrant unit.
Consider $\M_c\subseteq \M$ the full subcategory of cofibrant objects. Apply the above identification to $\C=\M_c$ to obtain the following equivalence in $\infty$-categories:
\[
\begin{tikzcd}
\N\Big( \comon(\M_c) \Big) \ar{r} & \coalginf_\ai( \N(\M_c))
\end{tikzcd}
\] 
Let $\W$ be the class of weak equivalences in $\M$.
By Proposition \ref{prop: compatibility of DK in sym mon} there is a symmetric monoidal functor $\N(\M_c^\o)\rightarrow \N(\M_c) \left[\W^{-1}\right]^\o$, which thus provides a map of $\infty$-categories:
\[
\begin{tikzcd}
\coalginf_\ai( \N(\M_c)) \ar{r} & \coalginf_\ai\left(\N(\M_c) \left[\W^{-1}\right]\right),
\end{tikzcd}
\]
and therefore we obtain a functor of $\infty$-categories:
\[
\begin{tikzcd}
\alpha:\N\Big(\comon(\M_c)\Big) \ar{r} & \coalginf_\ai\left(\N(\M_c) \left[\W^{-1}\right]\right).
\end{tikzcd}
\]
Denote $\W_\comon$ the class of morphisms in $\comon(\M_c)$ that are weak equivalences as underlying morphisms in $\M$. Notice that the above functor $\alpha$ sends $\W_\comon$ to equivalences.
By the universal property of Dwyer-Kan localizations as in Definition \ref{def: dk loc C[W-1] of any inf cat}, we obtain the following natural functor of $\infty$-categories:
\[
\begin{tikzcd}
\alpha:\N\Big(\comon(\M_c)\Big) \left[ \W_\comon^{-1} \right] \ar{r} & \coalginf_\ai\left(\N(\M_c) \left[\W^{-1}\right]\right).
\end{tikzcd}
\]
Similarly, for the cocommutative case we obtain the natural functor of $\infty$-categories:
\[
\begin{tikzcd}
\beta:\N\Big(\ccomon(\M_c)\Big) \left[ \W_\ccomon^{-1} \right] \ar{r} & \coalginf_\ei\left(\N(\M_c) \left[ \W^{-1} \right]\right).
\end{tikzcd}
\]
\begin{defi}\label{defi: coassoc and cocom rigidification}
Let $\M$ be a symmetric monoidal model category with cofibrant unit. Let $\alpha$ and $\beta$ be the functors described above.
If $\alpha$ is an equivalence of $\infty$-categories, we say that the model category $\M$ (or its Dwyer-Kan localization) \emph{ satisfies the coassociative rigidification}.
If $\beta$ is an equivalence of $\infty$-categories, we say that $\M$ (or its Dwyer-Kan localization) \emph{satisfies the cocommutative rigidification}.
\end{defi}

\begin{rem}
In general there is no reason to expect that if a model category $\M$ respects the associative rigidification then it also respects the coassociative rigidification. We see examples below in Corollaries \ref{cor: failure of rigi} and \ref{cor: non rigi  for coassociative}. Since coalgebras in $\M$ are algebras in $\M\op$ this might be surprising to the reader.
However, a key requirement for the associative rigidification is for $\M$ to be presentable, see \cite[4.1.8.4, 4.5.4.7]{lurie1}. But if $\M$ is presentable, then $\M\op$ is \emph{not} presentable, unless $\M$ is a complete lattice, see \cite[1.64]{Adamek-Rosicky}.
\end{rem}

\begin{rem}\label{rem: not easy to do rigidification in full generality}
If we inspect the dual case of algebras \cite[4.1.8.4, 4.5.4.7]{lurie1}, we see that we should have considered the $\infty$-category $\N(\comon(\M)) \left[ \W_\comon^{-1} \right]$ and not the $\infty$-category $\N(\comon(\M_c))\left[ \W_\comon^{-1} \right]$. There are several issues with that. 
\begin{itemize}

\item In general, these $\infty$-categories are not equivalent unless for instance $\M$ admits a functorial lax comonoidal cofibrant replacement. This means there is a functor $Q:\M\rightarrow \M_c$ such that there is a natural map: \[Q(X\otimes Y)\rightarrow Q(X)\otimes Q(Y),\] for any $X$ and $Y$ in $\M$. The main issue is that in general the functor $\N(\M_c^\o)\rightarrow \N(\M^\o)$ is only lax symmetric monoidal, see Remark \ref{rem:Dk SYM loc on cof or not is the same }.
Of course, if all objects in $\M$ are cofibrant, no such issues appear.

\item There is no guarantee to have a model structure on $\comon(\M)$ whose weak equivalences are $\W_\comon$, even when using the left-induced methods from \cite{left2}.
Even though we do not need a model category to define $\N(\comon(\M))\left[ \W_\comon^{-1} \right]$, this would help to determine if there was some compa\-tibility with $\M$.
For instance, if we suppose $\M$ is combinatorial monoidal model category and there exists a model structure on the category of coalgebra so that the forgetful-cofree adjunction:
\[
\begin{tikzcd}
U:\comon(\M) \ar[shift left=2]{r} & \ar[shift left=2]{l}[swap]{\perp} \M:\cofree,
\end{tikzcd}
\]
is a Quillen adjunction, then there exists a functorial cofibrant replacement $\comon(\M)\rightarrow \comon(\M_c)$ that induces an equivalence of $\infty$-categories:
\[
\N(\comon(\M_c))\left[ \W_\comon^{-1} \right] \simeq \N(\comon(\M))\left[ \W_\comon^{-1} \right].
\]
\item In the cases where $\comon(\M)$ does admit a model structure, it is in general left-induced by a model category that is not a monoidal model category. For instance, in chain complexes the left induced-model structure stem from the injective model structure (which is not monoidal) instead of the projective model structure, see \cite{left2}. See also Section \ref{sec: mod structure for coalg in sp} below.
\end{itemize}
All the above also applies to the cocommutative case.
\end{rem}

\subsection{The Cartesian Case}
We provide here a simple case of model categories satisfying the coassociative and cocommutative rigidification in the sense of Definition \ref{defi: coassoc and cocom rigidification}.
Let $(\M, \times, *)$ be a symmetric monoidal model category with respect to its Cartesian monoidal structure.  Let $\W$ be the class of weak equivalences in $\M$. Suppose the terminal object $*$ is cofibrant. %%Suppose also that $\M$ admits a functorial cofibrant replacement. 

\begin{prop}\label{prop: works for cartesian}
Let $(\M, \times, *)$ be as above. Then, $\M$ satisfies the coassociative and cocommutative rigidification, i.e. the following natural maps are equivalences of $\infty$-categories:
\[
\begin{tikzcd}
\N(\comon(\M_c)) \left[ \W_\comon^{-1} \right] \ar{r}{\simeq} & \coalginf_\ai\left(\N(\M_c) \left[ \W^{-1} \right]\right),
\end{tikzcd}
\]
\[
\begin{tikzcd}
\N(\ccomon(\M_c)) \left[ \W_\ccomon^{-1} \right] \ar{r}{\simeq} & \coalginf_\ei\left(\N(\M_c) \left[ \W^{-1} \right]\right),
\end{tikzcd}
\]
and all four of the $\infty$-categories above are equivalent to the Dwyer-Kan localization $\N(\M) \left[ \W^{-1} \right]$.
\end{prop}

\begin{proof}
For any Cartesian monoidal $\infty$-category $\Cinf$, we have the equivalence:
\[
\coalginf_\ai (\Cinf) \simeq \coalginf_\ei(\Cinf) \simeq \Cinf,
\]
see \cite[2.4.3.10]{lurie1}.
For any Cartesian monoidal (ordinary) category $\C$, we have the isomorphism of categories:
\[
\comon(\C) \cong \ccomon(\C) \cong \C.
\]
We obtain that the natural maps fit in the diagram:
\[
\begin{tikzcd}
\N(\comon(\M_c)) \left[ \W_\comon^{-1} \right] \ar{r}\ar{d}[swap]{\simeq} & \coalginf_\ai\left(\N(\M_c) \left[ \W^{-1} \right]\right)\ar{d}{\simeq}\\
\N(\M_c) \left[ \W^{-1} \right] \ar[equals]{r} & \N(\M_c) \left[ \W^{-1} \right]\\
\N(\ccomon(\M_c)) \left[ \W_\ccomon^{-1} \right] \ar{u}{\simeq}\ar{r} & \coalginf_\ei\left(\N(\M_c) \left[ \W^{-1} \right]\right)\ar{u}[swap]{\simeq}.
\end{tikzcd}
\]
Thus we obtain our desired equivalences.
\end{proof}

\section{Dold-Kan Correspondence For Coalgebras}
\label{section: DK loc for coalgebras}

We now apply Theorem \ref{thm: Weak monoidal eq imply strong in infinity} to a weak monoidal Quillen equivalence induced by the Dold-Kan correspondence as in \cite{monmodSS}, all missing details can be found there.
From now on, let $R$ be a commutative discrete ring.
Let $\mathsf{sMod}_R$ denote the category of simplicial $R$-modules, and let $\mathsf{Ch}^{\geq 0}_R$ denote the category of non-negatively graded chain complexes. 
The \emph{Dold-Kan correspondence} says that the \emph{normalization functor}:
\begin{equation}\label{eq: Dold-Kan equivalence}
\begin{tikzcd}
\mathsf{N}:\mathsf{sMod}_R \ar{r}{\cong} &  \mathsf{Ch}^{\geq 0}_R,
\end{tikzcd}
\end{equation}
is an equivalence of categories.
Its inverse functor is denoted $\Gamma: \mathsf{Ch}^{\geq 0}_R \rightarrow \mathsf{sMod}_R$.

We show in Corollary \ref{cor: dold kan colag} that if we derive the Dold-Kan correspondence, then we obtain a correspondence between the coalgebraic objects. Moreover, Theorem \ref{thm: Weak monoidal eq imply strong in infinity} clarifies the equivalence of the stable Dold-Kan correspondence, see Theorem \ref{cor: stable dold-kan eq}. Comparing this to a result of \cite{sore3}, we obtain that $\ai$-coalgebras of connective modules over a discrete commutative ring cannot be rigidified in the Dold-Kan context, see Section \ref{cor: failure of rigi}.

\subsection{The Derived Dold-Kan Equivalence}\label{DKloc: derived DK eq}
We endow each category with a model structure as follows. 
For $\mathsf{sMod}_R$, the weak equivalences and fibrations are the underlying weak equivalences and fibrations in simplicial sets, i.e., they are weak homotopy equivalences and Kan fibrations. 
For $\mathsf{Ch}^{\geq 0}_R$, we use the usual projective model structure. The weak equivalences are the quasi-isomorphisms, and the fibrations are the positive levelwise epimorphisms.
Both categories can be endowed with their usual symmetric monoidal structure induced by the tensor product of $R$-modules. However, the Dold-Kan equivalence (\ref{eq: Dold-Kan equivalence}) does \emph{not} preserve the monoidal structure.
Nonetheless, with respect to the above choice of model structures, the categories $\mathsf{sMod}_R$ and $\mathsf{Ch}^{\geq 0}_R$ are both symmetric monoidal model categories with cofibrant units.
The isomorphism of categories from (\ref{eq: Dold-Kan equivalence}) can be regarded now as either of two Quillen equivalences, depending on the choice of left and right adjoints:
\begin{equation}\label{eq: dold-kan sym mon}
\begin{tikzcd}[column sep=large]
\mathsf{Ch}^{\geq 0}_R \ar[shift left=2]{r}{\Gamma} & \ar[shift left=2]{l}{\mathsf{N}}[swap]{\perp} \mathsf{sMod}_R,
\end{tikzcd}
\end{equation}
and:
\begin{equation}\label{eq: dold-kan lax mon}
\begin{tikzcd}[column sep=large]
\mathsf{sMod}_R \ar[shift left=2]{r}{\mathsf{N}} & \ar[shift left=2]{l}{\Gamma}[swap]{\perp}\mathsf{Ch}^{\geq 0}_R .
\end{tikzcd}
\end{equation}
If we choose the normalization functor $\mathsf{N}$ to be the right adjoint as in (\ref{eq: dold-kan sym mon}), then it can be considered as lax symmetric monoidal via the shuffle map. 
If we choose $\mathsf{N}$ to be the left adjoint as in (\ref{eq: dold-kan lax mon}), then the Alexander-Whitney formula gives a lax comonoidal structure which is \emph{not} symmetric.

Nevertheless, this shows that both Quillen equivalences form a weak monoidal Quillen equivalence with cofibrant units, which is symmetric in the case where $\mathsf{N}$ is a right adjoint (\ref{eq: dold-kan sym mon}).
We can therefore apply Theorem \ref{thm: Weak monoidal eq imply strong in infinity} to obtain the following.

\begin{cor}[The Derived Dold-Kan Correspondence]\label{cor: derived dold-kan eq}
Let $R$ be a commutative discrete ring.
Then the Dwyer-Kan localizations of $\mathsf{sMod}_R$ and $\mathsf{Ch}^{\geq 0}_R$ are equivalent as symmetric monoidal $\infty$-categories:
\[
 \N\left(\mathsf{sMod}_R\right) \left[\W_\Delta^{-1}\right] \simeq \N\left(\mathsf{Ch}^{\geq 0}_R\right)\left[\W_\mathsf{dg}^{-1} \right],
\]
via the left Quillen derived functor $\Gamma: \mathsf{Ch}^{\geq 0}_R\rightarrow \mathsf{sMod}_R$ from the Quillen equivalence of \textup{(\ref{eq: dold-kan sym mon})},
where $\W_\Delta$ is the class of weak homotopy equivalences between simplicial $R$-modules, and $\W_\mathsf{dg}$ is the class of quasi-isomorphisms between non-negatively graded chain complexes over $R$.
\end{cor}

In particular, applying Remark \ref{rem: weak monoidal eq imply eq of coalgebras}, we get the following result.

\begin{cor}[Dold-Kan Correspondence For Coalgebras]\label{cor: dold kan colag}
For any $\infty$-operad $\Oinf$, the normalization functor induces an equivalence of $\infty$-categories:
\[
\begin{tikzcd}
\mathcal{CoAlg}_\Oinf\left( \N\left(\mathsf{sMod}_R\right) \left[\W_\Delta^{-1}\right]\right) \ar{r}{\simeq}&  \mathcal{CoAlg}_\Oinf\left(\N\left(\mathsf{Ch}^{\geq 0}_R\right)\left[\W_\mathsf{dg}^{-1} \right]\right).
\end{tikzcd}
\]
\end{cor}

\subsection{Rigidification Failure of Connective Coalgebras}\label{subsection: rigi failure in connective}
The above result bypasses a difficulty on the level of model categories and strict coalgebras.
If we choose the second adjunction (\ref{eq: dold-kan lax mon}) as a weak Quillen monoidal equivalence, then the normalization functor, being lax comonoidal, lifts to coalgebras: \[\mathsf{N}:\coalg_R \left(\mathsf{sMod}_R\right)\rightarrow \coalg_R ( \mathsf{Ch}^{\geq 0}_R ),\] 
but its inverse $\Gamma$, being only lax monoidal, does not lift to coalgebras. 
Nevertheless, a right adjoint exists on the level of $R$-coalgebras, either by presentability, or using dual methods as in section 3.3 of \cite{monmodSS}. We shall denote it by $\Gamma_\coalg$. 
Then, using left-induced methods of \cite{left2}, we obtain model structures and a Quillen adjunction:
\[
\begin{tikzcd}[column sep=large]
\coalg_R \left(\mathsf{sMod}_R\right) \ar[shift left =2]{r}{\mathsf{N}}[swap]{\perp} & \ar[shift left=2]{l}{\Gamma_\coalg} \coalg_R \left( \mathsf{Ch}^{\geq 0}_R \right).
\end{tikzcd}
\]
The weak equivalences are the underlying weak equivalences and every object is cofibrant, in both model categories.
However, it was shown in \cite[4.16]{sore3} that the above Quillen pair is \emph{not} a Quillen equivalence, at least when $R$ is a field. It was shown that for a particular choice of fibrant object $C$ in $\coalg_R ( \mathsf{Ch}^{\geq 0}_R)$, the counit $\mathsf{N}\left( \Gamma_\coalg \left( C \right) \right) \longrightarrow C$ is not a weak equivalence (i.e. not a quasi-isomorphism).

Therefore, the normalization functor does not induce an equivalence of $\infty$-categories on the Dwyer-Kan localizations:
\[ 
\begin{tikzcd}[column sep= small]
\mathsf{N}:\N\Big(\comon(\mathsf{sMod_R})\Big) \left[ \W_{\Delta, \mathsf{CoAlg}}^{-1} \right] \ar{r}{\not\simeq} & \N\Big(\comon(\mathsf{\ch}_R^{\geq 0})\Big) \left[ \W_{\mathsf{dg}, \mathsf{CoAlg}}^{-1} \right].
\end{tikzcd}
\]
Here $\W_{\Delta, \mathsf{CoAlg}}\subseteq \W_\Delta$ and $\W_{\mathsf{dg}, \mathsf{CoAlg}}\subseteq \W_\mathsf{dg}$ denote the subclasses of their respective weak equivalences between coalgebra objects.
However, Corollary \ref{cor: dold kan colag} shows that the normalization functor induces an equivalence between the $\mathbb{A}_\infty$-coalgebras:
\[
\begin{tikzcd}
\mathsf{N}:\mathcal{CoAlg}_{\mathbb{A}_\infty}\left( \N\left(\mathsf{sMod}_R\right) \left[\W_\Delta^{-1}\right]\right) \ar{r}{\simeq} &  \mathcal{CoAlg}_{\mathbb{A}_\infty}\left(\N\left(\mathsf{Ch}^{\geq 0}_R\right)\left[\W_\mathsf{dg}^{-1} \right]\right).
\end{tikzcd}
\]
By \cite[7.1.3.10]{lurie1}, the $\infty$-categories $\N(\mathsf{sMod}_R) [\W_\Delta^{-1}]$ and  $\N(\mathsf{Ch}^{\geq 0}_R)[\W_\mathsf{dg}^{-1}]$ represent the symmetric monoidal $\infty$-category $\mathcal{Mod}_{HR}^{\geq 0}$ of $HR$-modules in connective spectra. From our above discussion, we obtain the following result which says that we cannot rigidify coassociative coalgebras in one of the model categories $\ch_R^{\geq 0}$ or $\mathsf{sMod}_R$.

\begin{cor}\label{cor: failure of rigi}
Let $R$ be a commutative ring.
Consider $\mathsf{sMod}_R$ and $\mathsf{Ch}^{\geq 0}_R$ with their standard monoidal model structures.
Then one of these two categories, which we denote $\M$, does not satisfy the coassociative rigidification, i.e.:
\[
\begin{tikzcd}
\N(\coalg(\M_c)) \left[\W^{-1}\right] \ar{r}{\not\simeq} & \coalginf_{\mathbb{A}_\infty} \left(\mathcal{Mod}_{HR}^{\geq 0}\right),
\end{tikzcd}
\] 
where $\W$ is the class of maps of coalgebras which are weak equivalences in $\M$.
\end{cor}

\subsection{The Stable Dold-Kan Correspondence}

Our approach also gives a new proof of the \emph{stable Dold-Kan correspondence}. This well-known result was formalized with $\infty$-categories in {\cite[7.1.2.13]{lurie1}} as follows. Let $R$ be a commutative discrete ring. Then the $\infty$-category of $HR$-modules in spectra $\mathcal{Mod}_{HR}$ is equivalent to $\infty$-category of derived $R$-modules $\mathcal{D}(R)$ as symmetric monoidal $\infty$-categories: $\mathcal{Mod}_{HR} \simeq \mathcal{D}(R)$.
However, the equivalence was not described explicitly in \cite{lurie1}.
In \cite[2.10]{hzalgshipley}, Shipley provided an explicit zig-zag of (weak monoidal) Quillen equivalences between the standard model category $\mathsf{Mod}_{HR}$  of $HR$-modules in symmetric spectra and the projective model category of chain complexes over $R$:
\[
\begin{tikzcd}
\mathsf{Mod}_{HR}  \ar[shift left =2]{r}{}[swap]{\perp} & \ar[shift left=2]{l} \mathsf{Sp}^\Sigma\left(\mathsf{sMod}_R\right) \ar[shift left=2]{d} & \\
& \mathsf{Sp}^\Sigma\left(\mathsf{Ch}^{\geq 0}_R\right)\ar[shift left=2]{u}[swap]{\dashv} \ar[shift left =2]{r}{}[swap]{\perp} &  \ar[shift left=2]{l}\mathsf{Ch}_R.
\end{tikzcd}
\]
Notice that the Dwyer-Kan localizations of $\mathsf{Mod}_{HR}$ and $\mathsf{Ch}_R$ are precisely the $\infty$-categories $\mathcal{Mod}_{HR}$ and $\mathcal{D}(R)$ respectively.
If we derive and combine the Quillen functors above, we obtain a functor of $\infty$-categories $\Theta:\mathcal{Mod}_{HR} \rightarrow \mathcal{D}(R)$.
Recall that both $\mathsf{sMod}_R$ and $\mathsf{Ch}^{\geq 0}_R$ are left proper cellular symmetric monoidal model categories. 
Recall from \cite[B.3]{parsp} that there is a compatibility with stabilization and the Dwyer-Kan localization of a left proper cellular simplicial symmetric monoidal model category. Combining  the above with Corollary \ref{cor: derived dold-kan eq} and Theorem \ref{thm: Weak monoidal eq imply strong in infinity} yields the following.

\begin{cor}[The Stable Dold-Kan Correspondence]\label{cor: stable dold-kan eq}
Let $R$ be a commutative discrete ring. Then the $\infty$-category of $HR$-modules is equivalent to the $\infty$-category of derived  $R$-modules as symmetric monoidal $\infty$-categories via the functor:\[
\Theta:\mathcal{Mod}_{HR}\stackrel{\simeq}\longrightarrow \mathcal{D}(R).\]
Moreover, the functor induces an equivalence for any $\infty$-operad $\Oinf$:
\[
\coalginf_\Oinf\left(\mathcal{Mod}_{HR}\right)\stackrel{\simeq}\longrightarrow \coalginf_\Oinf\left(\mathcal{D}(R)\right).
\]
\end{cor}

\section{Coalgebras in Spectra}\label{sec: non rigi for spectra}
Based on the main result of \cite{perouxshipley}, we prove here, in Corollary \ref{cor: non rigi  for coassociative}, that the monoidal model categories of symmetric spectra (see \cite{SS}), orthogonal spectra (see \cite{MMSS} \cite{MM}), $\Gamma$-spaces (see \cite{Segal} \cite{BousfieldFrie}), and $\mathscr{W}$-spaces (see \cite{Anderson}), do not respect the coassociative rigidification, in the sense of Definition \ref{defi: coassoc and cocom rigidification}.
In other words, the strictly coassociative counital coalgebras in these monoidal categories of spectra do \emph{not} have the correct homotopy type. 

We work with the symmetric monoidal model category of symmetric spectra, denoted $\SpS$ (see \cite{SS}), and claim that similar results can be obtained with the other categories mentioned above, following \cite{perouxshipley}.
Notice that we have the equivalence of $\infty$-categories:
\[
\N(\SpS)[\mathsf{W}^{-1}] \simeq \mathcal{Sp},
\]
where $\W$ is the class of stable equivalences of symmetric spectra, and $\mathcal{Sp}$ is the $\infty$-category of spectra as in \cite[1.4.3.1]{lurie1}.

\subsection{Model Structures for Coalgebras}\label{sec: mod structure for coalg in sp}
Although not necessary to show the non-rigidification, as seen in Remark \ref{rem: not easy to do rigidification in full generality}, we provide a model category for coalgebras and cocommutative coalgebras in symmetric spectra.
We shall use the left-induced methods from \cite{left2}.
In \cite[Section 5]{SS} there is a simplicial, combinatorial model structure on $\SpS$ with all objects cofibrant called the \emph{(absolute) injective stable model structure}, see also \cite[Remark III.4.13]{schwede-book}. 
%%The fibrant objects are the injective $\Omega$-spectra.

\begin{prop}[{\cite[5.0.1, 5.0.2]{left2}}]
For any $\bS$-algebra $A$ in $\SpS$, there exists an injective model structure on $\Mod_A(\SpS)$ left-induced from the injective stable model structure on $\SpS$:
\[
\begin{tikzcd}[column sep=large]
\Mod_A(\SpS) \ar[shift left=2]{r}{U} & \ar[shift left=2]{l}{\Hom_{\SpS}(A,-)}[swap]{\perp} \SpS,
\end{tikzcd}
\]
with cofibrations the monomorphisms and weak equivalences the stable equivalences. This model structure on $\Mod_A(\SpS)$ is simplicial and combinatorial.
\end{prop}

Let $A$ be a commutative ring spectrum (i.e. a commutative $\bS$-algebra).
The symmetric monoidal category $(\Mod_A(\SpS), \sm_A, A)$ is presentable and the smash product $\sm_A$ preserves colimits in both variables. 
Thus we can apply \cite[2.7]{Porst} and we obtain the (forgetful-cofree)-adjunction between $A$-coalgebras and $A$-modules in $\SpS$:
\[
\begin{tikzcd}[column sep=large]
\coalg_A(\SpS) \ar[shift left=2]{r}{U} & \ar[shift left=2]{l}{\cofree}[swap]{\perp} \Mod_A(\SpS).
\end{tikzcd}
\]

\begin{prop} \label{prop: model cat for coalgebras in spectra }
Let $A$ be any commutative $\mathbb{S}$-algebra in symmetric spectra $\SpS$.
There exists a combinatorial model structure on $A$-coalgebras $\coalg_A(\SpS)$ left-induced by the (forgetful-cofree) adjunction from the injective stable model structure on $\Mod_A(\SpS)$. 
In particular, the weak equivalences in  $\coalg_A(\SpS)$ are the underlying stable equivalences, and the cofibrations are the underlying monomorphisms.
\end{prop}

\begin{proof}
We mimic the proof of \cite[Theorem 5.0.3]{left2}. We apply \cite[2.2.1]{left2}. Tensoring with a simplicial set lifts to $A$-coalgebras. Indeed, let $K$ be a simplicial set and $(C, \Delta_C, \varepsilon_C)$ be an $A$-coalgebra. Then the free $\bS$-module $\si_+K$ is endowed with a unique (cocommutative) $\bS$-coalgebra structure $(\si_+K, \Delta_K, \varepsilon_K)$, see \cite[Lemma 2.4]{perouxshipley}, where the comultiplication $\Delta_K$ is induced by the diagonal $K_+\r K_+\sm K_+$ and the counit $\varepsilon_K$ is induced by the non-trivial map $K_+\r S^0$.
Then the tensor $K\o C:=\Sigma^\infty_+ K\sm_\bS C$ is an $A$-coalgebra with comultiplication:
\begin{eqnarray*}
\si_+ K\sm_\bS C & \stackrel{\Delta_K\sm \Delta_C}\longrightarrow & (\si_+ K\sm_\bS \si_+ K) \sm_\bS (C\sm_A C) \\
& & \cong (\si_+K\sm_\bS C)\sm_A (\si_+ K\sm_\bS C),
\end{eqnarray*}
%%\begin{tikzcd}[column sep= large]
%%\si_+ K\sm_\bS C \ar{r}{\Delta_K\sm \Delta_C} & (\si_+ K\sm_\bS \si_+ K) \sm_\bS (C\sm_A C) \cong (\si_+K\sm_\bS C)\sm_A (\si_+ K\sm_\bS C),
%%\end{tikzcd}
and counit:
\[
\begin{tikzcd}
\si_+K\sm_\bS C \ar{r}{\varepsilon_K\sm \varepsilon_C} & \bS \sm_\bS A\cong A.
\end{tikzcd}
\]
There is a good cylinder object in $\sset$ given by the factorization:
\[
\begin{tikzcd}
\displaystyle S^0\coprod S^0 \ar[rightarrowtail]{r} & \Delta[1]_+=I\ar{r}{\sim} & S^0.
\end{tikzcd}
\]
Since $\Mod_A(\SpS)$ is simplicial, all objects are cofibrant, and that the smash product of an $A$-coalgebra with this factorization in $\sset$ lifts to $\coalg_A(\SpS)$, this defines a good cylinder object in $\coalg_A(\SpS)$ for any $A$-coalgebra $C$:
\[
\begin{tikzcd}
\displaystyle C\coprod C \ar[rightarrowtail]{r} & C\otimes I \ar{r}{\sim} & C,
\end{tikzcd}
\]
as $C\otimes S^0\cong C$, and colimits in $\coalg_A(\SpS)$ are computed in $\Mod_A(\SpS)$. This is a combinatorial model structure by \cite[2.23]{left1} and \cite[3.3.4]{left2}. 
\end{proof}

In the proof above, we see that the tensor $K\o C$ remains a cocommutative $A$-coalgebra if $C$ is cocommutative, since the coalgebra $\si_+ K$ is always cocommutative. Thus,
we can extend the result to cocommutative $A$-coalgebras.

\begin{prop} \label{prop: model cat for COCOMMUTATIVE coalgebras in spectra}
Let $A$ be any commutative $\mathbb{S}$-algebra in symmetric spectra $\SpS$.
There exists a combinatorial model structure on the category of cocommutative $A$-coalgebras $\ccoalg_A(\SpS)$ left-induced by the (forgetful-cofree) adjunction from the injective stable model structure on $\Mod_A(\SpS)$. 
In particular, the weak equivalences in  $\ccoalg_A(\SpS)$ are the underlying stable equivalences, and the cofibrations are the underlying monomorphisms.
\end{prop}

\subsection{The Failure of Rigidification}
We show here the failure of rigidification. Let $A$ and $B$ be commutative $\bS$-coalgebras. A map $A\rightarrow B$ is defined to be a \emph{positive flat cofibration of commutative $\bS$-algebras} if it is a cofibration in the model category of commutative $\bS$-algebras defined in \cite[3.2]{convenient} (or the positive flat stable model structure defined in \cite[III.6.1]{schwede-book}). Any $\mathbb{E}_\infty$-ring spectrum is equivalent to a positive flat cofibrant commutative $\bS$-algebra in $\SpS$.
As noted in \cite[2.4]{perouxshipley}, every comonoid in $(\mathsf{sSet}_*, \Smash, S^0)$ is of the form $Y_+$ and the comultiplication is given by the diagonal $Y_+\rightarrow (Y\times Y)_+\cong Y_+ \Smash Y_+$.

\begin{thm}[{\cite[3.4, 3.6]{perouxshipley}}]\label{thm: perouxshipley result}
Let $A$ be a positive flat cofibrant commutative $\mathbb{S}$-algebra in $\SpS$. Then, given any counital coassociative $A$-coalgebra $C$ in $\SpS$, the comultiplication is cocommutative and induced by the following epimorphism of $A$-coalgebras:
\[
A\Smash C_0 \longrightarrow C,
\]
where $A \Smash C_0$ is given an $A$-coalgebra structure via the diagonal on the pointed space $C_0\rightarrow C_0\Smash C_0$.
\end{thm}

Let $A$ be any commutative $\bS$-algebra. Let $\coalg_A(\SpS_c)$ denote the comonoid in the cofibrant objects of $A$-modules in $\SpS$ endowed with the absolute projective stable model structure (as in \cite[IV.6.1]{schwede-book}).  There is a natural map of $\infty$-categories:
\[
\begin{tikzcd}
 \alpha:\N(\mathsf{CoAlg}_A(\SpS_c)) \left[ \W^{-1}  \right] \ar{r} & \mathcal{CoAlg}_{\mathbb{A}_\infty}(\mathcal{Mod}_A(\mathcal{Sp})),
\end{tikzcd}
\]
where $\W$ is the class of stable equivalences between $A$-coalgebras. 

\begin{cor}\label{cor: non rigi  for coassociative}
Let $A$ be a positive flat cofibrant commutative $\bS$-algebra in $\SpS$. Then the $\infty$-category of $A$-modules $\mathcal{Mod}_A(\mathcal{Sp})$ does not satisfy the coassociative rigidification.
In parti\-cular, for $A=\bS$ we have:
\[
\begin{tikzcd}
\N(\coalg_\bS(\SpS_c)) \left[ \W^{-1}  \right] \ar{r}{\not\simeq} & \coalginf_{\mathbb{A}_\infty} (\mathcal{Sp}) 
\end{tikzcd}
\]
\end{cor}

\begin{proof}
Let $(C, \Delta, \varepsilon)$ be an $A$-coalgebra in $\SpS$ that is cofibrant as an $A$-module in the (absolute) projective stable model structure. 
Suppose the functor:
\[
\begin{tikzcd}
 \alpha:\N(\mathsf{CoAlg}_A(\SpS_c)) \left[ \W^{-1}  \right] \ar{r} & \mathcal{CoAlg}_{\mathbb{A}_\infty}(\mathcal{Mod}_A(\mathcal{Sp})),
\end{tikzcd}
\]
is an equivalence of $\infty$-categories.
By Theorem \ref{thm: perouxshipley result}, we see that $\alpha(C)$ is automatically an $\ei$-coalgebra.
But there exist $\ai$-coalgebras in $\mathcal{Sp}$ that are not $\ei$-coalgebras. 
Indeed, take any compact topological group that is not Abelian (say $O(2)$), then $A\wedge O(2)_+$ is an $\ai$-algebra in $\mathcal{Mod}_A(\mathcal{Sp})$ that is not commutative and is a compact spectrum. By Spanier-Whitehead duality, we obtain an $\ai$-coalgebra that is not $\ei$ in spectra.
\end{proof}

\begin{rem}
Let $A$ be a positive flat cofibrant commutative $\bS$-algebra in $\SpS$. Let $G$ be a compact Abelian Lie group (say a torus). Then the Spanier-Whitehead dual of the commutative ring spectrum $A\wedge G_+$ is an $\ei$-coalgebra in spectra with comultiplication that is not the diagonal.
Therefore the $\infty$-category of $A$-modules $\mathcal{Mod}_A(\mathcal{Sp})$ does not satisfy the cocommutative rigidification.
In parti\-cular, for $A=\bS$ we have:
\[
\begin{tikzcd}
\N(\ccoalg_\bS(\SpS_c)) \left[ \W^{-1}  \right] \ar{r}{\not \simeq} & \coalginf_{\ei} (\mathcal{Sp}).
\end{tikzcd}
\]
\end{rem}

\renewcommand{\bibname}{References}
\bibliographystyle{amsalpha}
\bibliography{biblio}

\newcommand{\etalchar}[1]{$^{#1}$}
\providecommand{\bysame}{\leavevmode\hbox to3em{\hrulefill}\thinspace}
\providecommand{\MR}{\relax\ifhmode\unskip\space\fi MR }
% \MRhref is called by the amsart/book/proc definition of \MR.
\providecommand{\MRhref}[2]{%
  \href{http://www.ams.org/mathscinet-getitem?mr=#1}{#2}
}
\providecommand{\href}[2]{#2}
\begin{thebibliography}{EKMM97}

\bibitem[ABG18]{parsp}
Matthew Ando, Andrew~J. Blumberg, and David Gepner, \emph{Parametrized spectra,
  multiplicative {T}hom spectra and the twisted {U}mkehr map}, Geom. Topol.
  \textbf{22} (2018), no.~7, 3761--3825. \MR{3890766}

\bibitem[And74]{Anderson}
D.~W. Anderson, \emph{Convergent functors and spectra}, Localization in group
  theory and homotopy theory, and related topics ({S}ympos., {B}attelle
  {S}eattle {R}es. {C}enter, {S}eattle, {W}ash., 1974), Lecture Notes in Math.,
  vol. 418, Springer, Berlin, 1974, pp.~1--5. \MR{0383388}

\bibitem[AR94]{Adamek-Rosicky}
Ji\v{r}\'\i{} Ad\'amek and Ji\v{r}\'\i{} Rosick\'y, \emph{Locally presentable
  and accessible categories}, London Mathematical Society Lecture Note Series,
  vol. 189, Cambridge University Press, Cambridge, 1994. \MR{1294136}

\bibitem[Ber07]{bergner}
Julia~E. Bergner, \emph{A model category structure on the category of
  simplicial categories}, Trans. Amer. Math. Soc. \textbf{359} (2007), no.~5,
  2043--2058. \MR{2276611}

\bibitem[BF78]{BousfieldFrie}
A.~K. Bousfield and E.~M. Friedlander, \emph{Homotopy theory of {$\Gamma
  $}-spaces, spectra, and bisimplicial sets}, Geometric applications of
  homotopy theory ({P}roc. {C}onf., {E}vanston, {I}ll., 1977), {II}, Lecture
  Notes in Math., vol. 658, Springer, Berlin, 1978, pp.~80--130. \MR{513569}

\bibitem[BHK{\etalchar{+}}15]{left1}
Marzieh Bayeh, Kathryn Hess, Varvara Karpova, Magdalena K\c{e}dziorek, Emily
  Riehl, and Brooke Shipley, \emph{Left-induced model structures and diagram
  categories}, Women in topology: collaborations in homotopy theory, Contemp.
  Math., vol. 641, Amer. Math. Soc., Providence, RI, 2015, pp.~49--81.
  \MR{3380069}

\bibitem[DK80]{dwyer-kan}
W.~G. Dwyer and D.~M. Kan, \emph{Calculating simplicial localizations}, J. Pure
  Appl. Algebra \textbf{18} (1980), no.~1, 17--35. \MR{578563}

\bibitem[EKMM97]{EKMM}
Anthony Elmendorf, Igor Kriz, Michael~A. Mandell, and Peter May, \emph{Rings,
  modules, and algebras in stable homotopy theory}, Mathematical Surveys and
  Monographs, vol.~47, American Mathematical Society, Providence, RI, 1997,
  With an appendix by M. Cole. \MR{1417719}

\bibitem[{Gro}20]{groth}
Moritz {Groth}, \emph{{A short course on $\infty$-categories}}, Handbook of
  homotopy theory (Haynes Miller, ed.), CRC Press, Boca Raton, FL, 2020,
  pp.~549--618.

\bibitem[Hin16]{hin}
Vladimir Hinich, \emph{Dwyer-{K}an localization revisited}, Homology Homotopy
  Appl. \textbf{18} (2016), no.~1, 27--48. \MR{3460765}

\bibitem[HKRS17]{left2}
Kathryn Hess, Magdalena K\c{e}dziorek, Emily Riehl, and Brooke Shipley, \emph{A
  necessary and sufficient condition for induced model structures}, J. Topol.
  \textbf{10} (2017), no.~2, 324--369. \MR{3653314}

\bibitem[Hov99]{hovey}
Mark Hovey, \emph{Model categories}, Mathematical Surveys and Monographs,
  vol.~63, American Mathematical Society, Providence, RI, 1999. \MR{1650134}

\bibitem[HSS00]{SS}
Mark Hovey, Brooke Shipley, and Jeff Smith, \emph{Symmetric {S}pectra}, J.
  Amer. Math. Soc. \textbf{13} (2000), no.~1, 149--208. \MR{1695653}

\bibitem[Lur09]{htt}
Jacob Lurie, \emph{Higher topos theory}, Annals of Mathematics Studies, vol.
  170, Princeton University Press, Princeton, NJ, 2009. \MR{2522659}

\bibitem[Lur17]{lurie1}
Jacob Lurie, \emph{Higher algebra},
  \url{https://www.math.ias.edu/~lurie/papers/HA.pdf}, 2017, electronic book.

\bibitem[Man03]{mandell}
Michael~A. Mandell, \emph{Topological {A}ndr\'{e}-{Q}uillen cohomology and
  {$E_\infty$}, {A}ndr\'{e}-{Q}uillen cohomology}, Adv. Math. \textbf{177}
  (2003), no.~2, 227--279. \MR{1990939}

\bibitem[MM02]{MM}
Michael~A. Mandell and Peter May, \emph{Equivariant orthogonal spectra and
  {$S$}-modules}, Mem. Amer. Math. Soc. \textbf{159} (2002), no.~755, x+108.
  \MR{1922205}

\bibitem[MMSS01]{MMSS}
Michael~A. Mandell, Peter May, Stefan Schwede, and Brooke Shipley, \emph{Model
  categories of diagram spectra}, Proc. London Math. Soc. (3) \textbf{82}
  (2001), no.~2, 441--512. \MR{1806878}

\bibitem[NS18]{tch}
Thomas Nikolaus and Peter Scholze, \emph{On topological cyclic homology}, Acta
  Math. \textbf{221} (2018), no.~2, 203--409. \MR{3904731}

\bibitem[P{\'e}r20a]{coalgenr}
Maximilien P{\'e}roux, \emph{The coalgebraic enrichment of algebras in higher
  categories}, 2020, arXiv.

\bibitem[P{\'e}r20b]{phd}
\bysame, \emph{Highly structured coalgebras and comodules},
  \url{https://homepages.math.uic.edu/~mholmb2/thesis-2.pdf}, 2020, Phd thesis.

\bibitem[P{\'e}r20c]{connectivecomod}
\bysame, \emph{Rigidificaton of connective comodules}, 2020, arXiv.

\bibitem[Por08]{Porst}
Hans-E. Porst, \emph{On categories of monoids, comonoids, and bimonoids},
  Quaest. Math. \textbf{31} (2008), no.~2, 127--139. \MR{2529129}

\bibitem[PS19]{perouxshipley}
Maximilien P\'{e}roux and Brooke Shipley, \emph{Coalgebras in symmetric
  monoidal categories of spectra}, Homology Homotopy Appl. \textbf{21} (2019),
  no.~1, 1--18. \MR{3852287}

\bibitem[RS17]{richtershipley}
Birgit Richter and Brooke Shipley, \emph{An algebraic model for commutative
  {$H\Bbb{Z}$}-algebras}, Algebr. Geom. Topol. \textbf{17} (2017), no.~4,
  2013--2038. \MR{3685600}

\bibitem[Sch]{schwede-book}
Stefan Schwede, \emph{Symmetric spectra},
  \url{http://www.math.uni-bonn.de/people/schwede/SymSpec-v3.pdf}, electronic
  book.

\bibitem[Seg74]{Segal}
Graeme Segal, \emph{Categories and cohomology theories}, Topology \textbf{13}
  (1974), 293--312. \MR{0353298}

\bibitem[Shi04]{convenient}
Brooke Shipley, \emph{A convenient model category for commutative ring
  spectra}, Homotopy theory: relations with algebraic geometry, group
  cohomology, and algebraic {$K$}-theory, Contemp. Math., vol. 346, Amer. Math.
  Soc., Providence, RI, 2004, pp.~473--483. \MR{2066511}

\bibitem[Shi07]{hzalgshipley}
\bysame, \emph{{$H\Bbb Z$}-algebra spectra are differential graded algebras},
  Amer. J. Math. \textbf{129} (2007), no.~2, 351--379. \MR{2306038}

\bibitem[Sor19]{sore3}
W.~Hermann~B. Sor\'{e}, \emph{On a {Q}uillen adjunction between the categories
  of differential graded and simplicial coalgebras}, J. Homotopy Relat. Struct.
  \textbf{14} (2019), no.~1, 91--107. \MR{3913972}

\bibitem[SS03]{monmodSS}
Stefan Schwede and Brooke Shipley, \emph{Equivalences of monoidal model
  categories}, Algebr. Geom. Topol. \textbf{3} (2003), 287--334. \MR{1997322}

\end{thebibliography}
\end{document}